\documentclass[reqno,12pt]{amsart}
\usepackage{amscd,amsfonts,amssymb}
\numberwithin{equation}{section}
\newtheorem{Theorem}{Theorem}[section]
\newtheorem{Lemma}{Lemma}[section]
\newtheorem{Corollary}{Corollary}[section]

\theoremstyle{definition}
\newtheorem{Definition}{Definition}[section]
\theoremstyle{remark}

\newtheorem{Example}{Example}[section]
\newcommand{\sign}{\mathop{\rm sign}\nolimits}

\author{Andrej A. Kon'kov}
\address{Department of Differential Equations,
Faculty of Mechanics and Mathematics,
Mo\-s\-cow Lo\-mo\-no\-sov State University,
Vorobyovy Gory,
119992 Moscow, Russia}
\email{konkov@mech.math.msu.su}
\title[]{On stabilization of solutions of nonlinear parabolic equations with a gradient term}
\thanks{The research was supported by RFBR, grant 11-01-12018-ofi-m-2011.}
\date{}

\begin{document}

\begin{abstract}

For parabolic equations of the form
$$
	\frac{\partial u}{\partial t}
	-
	\sum_{i,j=1}^n
	a_{ij} (x, u)
	\frac{\partial^2 u}{\partial x_i \partial x_j}
	+
	f (x, u, D u)
	=
	0
	\quad
	\mbox{in } {\mathbb R}_+^{n+1},
$$
where 
${\mathbb R}_+^{n+1} = {\mathbb R}^n \times (0, \infty)$, $n \ge 1$,
$D = (\partial / \partial x_1, \ldots, \partial / \partial x_n)$
is the gradient operator, and $f$ is some function, 
we obtain conditions guaranteeing that every solution tends to zero as $t \to \infty$.
\end{abstract}

\maketitle

\section{Introduction}\label{intro}
We study solutions of the equations
\begin{equation}
	\frac{\partial u}{\partial t}
	-
	\sum_{i,j=1}^n
	a_{ij} (x, u)
	\frac{\partial^2 u}{\partial x_i \partial x_j}
	+
	f (x, u, D u)
	=
	0
	\quad
	\mbox{in } {\mathbb R}_+^{n+1},
	\label{1.1}
\end{equation}
where 
${\mathbb R}_+^{n+1} = {\mathbb R}^n \times (0, \infty)$, $n \ge 1$,
and
$D = (\partial / \partial x_1, \ldots, \partial / \partial x_n)$
is the gradient operator.
We assume that 
$$
	\sum_{i,j=1}^n
	a_{ij} (x, \zeta)
	\xi_i
	\xi_j
	>
	0
$$
for all $x \in {\mathbb R}^n$, $\zeta \in {\mathbb R} \setminus \{ 0 \}$, and 
$\xi = (\xi_1, \ldots, \xi_n) \in {\mathbb R}^n \setminus \{ 0 \}$.
Also let there are locally bounded measurable functions
$g : (0, \infty) \to (0, \infty)$,
$h : (0, \infty) \to (0, \infty)$,
and
$p : {\mathbb R}^n \to [0, \infty)$
such that
$$
	\inf_K
	g
	>
	0
	\quad
	\mbox{and}
	\quad
	\inf_K
	h
	>
	0
$$
for any compact set $K \subset (0, \infty)$ and, moreover,
\begin{equation}
	f (x, \zeta, \lambda x) \sign \zeta
	\ge
	p (x)
	g (|\zeta|)
	\left(
		1
		+
		\sum_{i,j=1}^n
		|a_{ij} (x, \zeta)|
	\right)
	\label{1.2}
\end{equation}
for all $x \in {\mathbb R}^n$, $\zeta \in {\mathbb R} \setminus \{ 0 \}$,
and $0 \le \lambda \le p (x) h (|\zeta|)$,
where 
$$
	\sign \zeta
	=
	\left\{
		\begin{array}{rl}
		1,
		&
		\zeta > 0,
		\\
		0,
		&
		\zeta = 0,
		\\
		- 1,
		&
		\zeta < 0.
	\end{array}
	\right.
$$

By a solution of~\eqref{1.1} we mean a function $u$ that has two continuous derivatives with respect to $x$ and one continuous derivative with respect to $t$ and satisfies equation~\eqref{1.1} in the classical sense~\cite{Landis}.

No smoothness assumptions on $a_{ij}$ and $f$ are imposed, 
we do not even require these functions to be measurable.

Let us denote 
$B_r^x = \{ y \in {\mathbb R}^n : |y - x| < r \}$,
$S_r^x = \{ y \in {\mathbb R}^n : |y - x| = r \}$,
and
$Q_{x, r}^{t_1, t_2} = B_r^x \times (t_1, t_2)$.
In the case of $x = 0$, we write $B_r$, $S_r$, and $Q_r^{t_1, t_2}$ 
instead of $B_r^0$, $S_r^0$, and $Q_{0, r}^{t_1, t_2}$, respectively.

Put
$$
	q (r)
	=
	\inf_{B_r}
	p,
	\quad
	r \in (0, \infty).
$$

For any function $\varphi : (0, \infty) \to {\mathbb R}$ and a real number $\theta > 1$ we also denote
$$
	\varphi_\theta (\zeta)
	=
	\inf_{(\zeta / \theta, \theta \zeta)}
	\varphi,
	\quad
	\zeta \in (0, \infty).
$$

The questions treated in this paper were investigated earlier by a number 
of authors~[1--7, 10, 11].
Below, we obtain conditions guaranteeing that every solution of~\eqref{1.1} tends
to zero as $t \to \infty$. 
These conditions take into account the dependence of the function $f$ on the gradient term $D u$.
Our results are applicable to a wide class of nonlinear equations 
(see Examples~\ref{E2.1}--\ref{E2.4}).

\section{Main results}

\begin{Theorem}\label{T2.1} Let
\begin{equation}
	\int_1^\infty
	r q (r)
	\,
	dr
	=
	\infty
	\label{T2.1.1}
\end{equation}
and, moreover,
\begin{equation}
	\int_1^\infty
	(g_\theta (\zeta) \zeta)^{- 1/2}
	\,
	d\zeta
	<
	\infty
	\label{T2.1.2}
\end{equation}
and
\begin{equation}
	\int_1^\infty
	\frac{
		d\zeta
	}{
		h_\theta (\zeta)
	}
	<
	\infty
	\label{T2.1.3}
\end{equation}
for some real number $\theta > 1$. 
Then any solution of~\eqref{1.1} stabilizes to zero uniformly
on an arbitrary compact set $K \subset {\mathbb R}^n$ as $t \to \infty$,
i.e.
$$
	\lim_{t \to \infty}
	\sup_{x \in K}
	|u (x, t)|
	=
	0.
$$
\end{Theorem}

Theorem~\ref{T2.1} will be proved later. Now, we demonstrate its application.

\begin{Example}\label{E2.1}
Consider the equation
\begin{equation}
	u_t
	=
	\Delta u 
	+ 
	b (x, u, D u)
	-
	c (x) |u|^{\sigma - 1} u
	\quad
	\mbox{in } {\mathbb R}_+^{n+1},
	\label{E2.1.1}
\end{equation}
where
\begin{equation}
	|b (x, \zeta, \xi)|
	\le
	b_0 (1 + |x|)^k \zeta^\mu |\xi|^\alpha,
	\quad
	b_0 = const > 0,
	\label{E2.1.2}
\end{equation}
and
\begin{equation}
	c (x) \ge c_0 (1 + |x|)^l,
	\quad
	c_0 = const > 0,
	\label{E2.1.3}
\end{equation}
for all 
$x \in {\mathbb R}^n$, 
$\zeta \in {\mathbb R} \setminus \{ 0 \}$, 
and 
$\xi \in {\mathbb R}^n$.
We assume that
$\alpha > 0$ whereas $\mu$, $\sigma$, $k$, and $l$ can be arbitrary real numbers.

According to Theorem~\ref{T2.1}, if
\begin{equation}
	\min 
	\{ 
		l - k + \alpha,
		l + 2
	\}
	\ge
	0
	\label{E2.1.4}
\end{equation}
and
\begin{equation}
	\sigma > \max \{ 1, \alpha + \mu \},
	\label{E2.1.5}
\end{equation}
then any solution of~\eqref{E2.1.1} stabilizes to zero uniformly on an arbitrary compact 
subset of ${\mathbb R}^n$ as $t \to \infty$.
Really, taking
$f (x, \zeta, \xi) = c (x) |\zeta|^{\sigma - 1} \zeta - b (x, \zeta, \xi)$,
$g (\zeta) = \zeta^\sigma$,
and 
$h (\zeta) = \zeta^{(\sigma - \mu) / \alpha}$, 
we obtain that~\eqref{1.2} is valid with 
\begin{equation}
	p (x)
	=
	p_0
	(1 + |x|)^{
		\min 
		\{ 
			(l - k) / \alpha - 1,
			\,
			l
		\}
	},
	\label{E2.1.6}
\end{equation}
where $p_0 > 0$ is a sufficiently small real number.
In so doing, relation~\eqref{T2.1.1} is equivalent to~\eqref{E2.1.4} 
while~\eqref{T2.1.2} and~\eqref{T2.1.3} are equivalent to~\eqref{E2.1.5}.

At the same time, if~\eqref{E2.1.4} is not fulfilled, 
then for some functions $b$ and $c$ satisfying~\eqref{E2.1.2} and~\eqref{E2.1.3}
equation~\eqref{E2.1.1} has a positive solution which does not depend on the variable $t$.
This solution obviously can not stabilize to zero as $t \to \infty$.
In turn, if~\eqref{E2.1.5} is not fulfilled, 
then for all real numbers $k$ and $l$ there exist
functions $b$ and $c$ such that~\eqref{E2.1.2} and~\eqref{E2.1.3} are valid 
and equation~\eqref{E2.1.1} has a positive solution independent of $t$. 
In this sense, conditions~\eqref{E2.1.4} and~\eqref{E2.1.5} are the best possible.

We also note that~\eqref{E2.1.4} and~\eqref{E2.1.5}
correspond to the blow-up conditions for non-negative solutions of the elliptic inequalities
$$
	\Delta u 
	+ 
	b (x, u, D u)
	\ge
	c (x) u^\sigma
	\quad
	\mbox{in } {\mathbb R}^n
$$
considered in ~\cite[Example~2.1]{meNONANA} for $\mu = 0$ and $1 \le \alpha \le 2$.
\end{Example}

\begin{Example}\label{E2.2}
In~\eqref{E2.1.1}, let the functions $b$ and $c$ satisfy the relations
\begin{equation}
	|b (x, \zeta, \xi)|
	\le
	b_0 (1 + |x|)^k \log^s (2 + |x|) \zeta^\mu |\xi|^\alpha,
	\quad
	b_0 = const > 0,
	\label{E2.2.1}
\end{equation}
and
\begin{equation}
	c (x) \ge c_0 (1 + |x|)^l \log^m (2 + |x|),
	\quad
	c_0 = const > 0,
	\label{E2.2.2}
\end{equation}
for all $x \in {\mathbb R}^n$, $\zeta \in {\mathbb R} \setminus \{ 0 \}$, and $\xi \in {\mathbb R}^n$,
where $\alpha$, $k$, $s$, $l$, and $m$ are real numbers with $\alpha > 0$ and
$$
	\min 
	\{ 
		l - k + \alpha,
		l + 2
	\}
	=
	0.
$$
In other words, we consider the case of the critical exponents $l$ and $k$ in~\eqref{E2.1.4}.

As in the previous example, we take 
$f (x, \zeta, \xi) = c (x) |\zeta|^{\sigma - 1} \zeta - b (x, \zeta, \xi)$,
$g (\zeta) = \zeta^\sigma$,
and 
$h (\zeta) = \zeta^{(\sigma - \mu) / \alpha}$.
It can be verified that~\eqref{1.2} is valid for 
$$
	p (x)
	=
	p_0
	(1 + |x|)^{-2}
	\log^\gamma
	(2 + |x|),
$$
where $p_0 > 0$ is a sufficiently small real number and
$$
	\gamma
	=
	\left\{
		\begin{aligned}
			&
			m,
			&
			&
			l + 2 < l - k + \alpha,
			\\
			&
			\min 
			\{ 
				(m - s) / \alpha, 
				\, 
				m 
			\},
			&
			&
			l + 2 = l - k + \alpha,
			\\
			&
			(m - s) / \alpha,
			&
			&
			l + 2 > l - k + \alpha.
		\end{aligned}
	\right.
$$
Thus, by Theorem~\ref{T2.1}, if~\eqref{E2.1.5} holds and
\begin{equation}
	\gamma
	\ge
	- 1,
	\label{E2.2.3}
\end{equation}
then any solution of~\eqref{E2.1.1} stabilizes to zero 
uniformly on an arbitrary compact subset of ${\mathbb R}^n$ as $t \to \infty$.

If~\eqref{E2.2.3} is not fulfilled and $n \ge 3$, 
then there exist functions $b$ and $c$ such that~\eqref{E2.2.1} and~\eqref{E2.2.2}
are valid and~\eqref{E2.1.1} has a positive solution independent of $t$.
Condition~\eqref{E2.1.5} is also the best possible.
\end{Example}

\begin{Example}\label{E2.3}
Consider the equation
\begin{equation}
	u_t
	=
	\Delta u 
	+ 
	b (x, u, D u)
	-
	c (x) 
	|u|^{\sigma - 1} u
	\log^\nu (1 + |u|)
	\quad
	\mbox{in } {\mathbb R}_+^{n+1},
	\label{E2.3.1}
\end{equation}
where the functions $b$ and $c$ satisfy~\eqref{E2.1.2} and~\eqref{E2.1.3} 
with $\alpha > 0$ and
$$
	\sigma = \max \{ 1, \alpha + \mu \},
$$
i.e. we are interested in the case of the critical exponent $\sigma$ in~\eqref{E2.1.5}.

Taking
$f (x, \zeta, \xi) = c (x) |\zeta|^{\sigma - 1} \zeta \log^\nu (1 + \zeta) - b (x, \zeta, \xi)$,
$g (\zeta) = \zeta^\sigma \log^\nu (1 + \zeta)$, 
and
$h (\zeta) = \zeta^{(\sigma - \mu) / \alpha} \log^{\nu / \alpha} (1 + \zeta)$,
one can see that~\eqref{1.2} holds with some function $p$ of the form~\eqref{E2.1.6}.
Thus, in accordance with Theorem~\ref{T2.1} if~\eqref{E2.1.4} is valid and
\begin{equation}
	\nu
	>
	\left\{
		\begin{aligned}
			&
			2,
			&
			&
			\alpha + \mu < 1,
			\\
			&
			\max \{ 2, \alpha \},
			&
			&
			\alpha + \mu = 1,
			\\
			&
			\alpha,
			&
			&
			\alpha + \mu > 1,
		\end{aligned}
	\right.
	\label{E3.1.3}
\end{equation}
then any solution of~\eqref{E2.3.1} stabilizes to zero 
uniformly on an arbitrary compact subset of ${\mathbb R}^n$ as $t \to \infty$.

If~\eqref{E2.1.4} is not fulfilled, then there are functions $b$ and $c$ 
such that~\eqref{E2.1.2} and~\eqref{E2.1.3} hold and equation~\eqref{E2.3.1} 
has a positive solution independent of the variable $t$.
In turn, if~\eqref{E3.1.3} is not fulfilled, 
then for all real numbers $k$ and $l$
there exist functions $b$ and $c$ satisfying~\eqref{E2.1.2} and~\eqref{E2.1.3} 
for which~\eqref{E2.3.1} has a positive solution independent of $t$.
\end{Example}

\begin{Example}\label{E2.4}
In the equation
\begin{equation}
	u_t
	=
	\Delta u 
	+ 
	b (x, D u)
	-
	c (x, u)
	\quad
	\mbox{in } {\mathbb R}_+^{n+1},
	\label{E2.4.1}
\end{equation}
let
\begin{equation}
	|b (x, \xi)|
	\le
	\varphi (|\xi|)
	\label{E2.4.2}
\end{equation}
and
\begin{equation}
	c (x, \zeta) \sign \zeta
	\ge
	\psi (|\zeta|)
	\label{E2.4.3}
\end{equation}
for all $x \in {\mathbb R}^n$, $\zeta \in {\mathbb R} \setminus \{ 0 \}$, and $\xi \in {\mathbb R}^n$,
where 
$\varphi : [0, \infty) \to [0, \infty)$
and
$\psi : (0, \infty) \to (0, \infty)$
are non-decreasing continuous functions.
We also assume that $\varphi : [0, \infty) \to [0, \infty)$ is a bijection and
$$
	\liminf_{\zeta \to \infty}
	\frac{
		\varphi^{-1} (\zeta)
	}{
		\varphi^{-1} (2 \zeta)
	}
	>
	0,
$$
where $\varphi^{-1}$ is the inverse function of $\varphi$.

According to Theorem~\ref{T2.1}, if
\begin{equation}
	\int_1^\infty
	(\psi (\zeta) \zeta)^{- 1/2}
	\,
	d\zeta
	<
	\infty
	\label{E2.4.6}
\end{equation}
and
\begin{equation}
	\int_1^\infty
	\frac{
		d\zeta
	}{
		\varphi^{-1} \circ \psi (\zeta)
	}
	<
	\infty,
	\label{E2.4.7}
\end{equation}
then any solution of~\eqref{E2.4.1} stabilizes to zero 
uniformly on an arbitrary compact subset of ${\mathbb R}^n$ as $t \to \infty$.
Indeed, we put 
$
	f (x, \zeta, \xi) = c (x, \zeta) - b (x, \xi).
$
It does not present any particular problem to verify that~\eqref{1.2} is valid with
$
	g (\zeta) = \psi (\zeta),
$
$
	h (\zeta) = \varphi^{-1} (\varepsilon \psi (\zeta)),
$
and
$
	p (x) = \varepsilon,
$
where $\varepsilon > 0$ is a sufficiently small real number.
In so doing, condition~\eqref{T2.1.1} is certainly satisfied
while~\eqref{T2.1.2} and~\eqref{T2.1.3} are equivalent to~\eqref{E2.4.6} and~\eqref{E2.4.7},
respectively.

We can show that, 
in the case where at least one of conditions~\eqref{E2.4.6}, \eqref{E2.4.7} is not fulfilled, 
there are functions $b$ and $c$ such that~\eqref{E2.4.2} and~\eqref{E2.4.3} hold and
equation~\eqref{E2.4.1} has a positive solution which does not depend on $t$.
\end{Example}

Proof of Theorem~\ref{T2.1} relies on the following assertion.

\begin{Theorem}\label{T2.2}
Suppose that $u$ is a solution of the inequality
\begin{equation}
	\frac{\partial u}{\partial t}
	-
	\sum_{i,j=1}^n
	a_{ij} (x, u)
	\frac{\partial^2 u}{\partial x_i \partial x_j}
	+
	f (x, u, D u)
	\le
	0
	\quad
	\mbox{in } {\mathbb R}_+^{n+1}.
	\label{T2.2.1}
\end{equation}
Also let there exist a real number $\theta > 1$ 
such that~\eqref{T2.1.1}, \eqref{T2.1.2}, and~\eqref{T2.1.3} hold.
Then
$$
	\lim_{t \to \infty}
	\sup_{x \in K}
	u_+ (x, t)
	=
	0
$$
for any compact set $K \subset {\mathbb R}^n$, where
$$
	u_+ (x, t)
	=
	\left\{
		\begin{aligned}
			&
			u (x, t),
			&
			&
			u (x, t) > 0,
			\\
			&
			0,
			&
			&
			u (x, t) \le 0.
		\end{aligned}
	\right.
$$
\end{Theorem}

Proof of Theorem~\ref{T2.2} is given in Section~\ref{proof}.
As for equation~\eqref{1.1}, solutions of~\eqref{T2.2.1} 
are understood in the classical sense.

\begin{proof}[Proof of Theorem~$\ref{T2.1}$]
We apply Theorem~\ref{T2.2} to the functions $u$ and $-u$.
\end{proof}

\section{Proof of Theorem~\ref{T2.2}}\label{proof}

In this section we assume that hypotheses of Theorem~\ref{T2.2} are satisfied. Let us denote
$$
	L 
	=
	\sum_{i,j=1}^n
	a_{ij} (x, u)
	\frac{\partial^2}{\partial x_i \partial x_j}.
$$
By $C$ we mean various positive constants which can depend only on $n$ and $\theta$.

\begin{Definition}[\cite{Landis}]\label{D3.1}
A point $(x, t)$ belongs to the upper lid $\gamma (\Omega)$ of an open set 
$\Omega \subset {\mathbb R}_+^{n+1}$ 
if there exist real numbers $r > 0$ and $\varepsilon > 0$ such that 
$Q_{x, r}^{t - \varepsilon, t} \subset \Omega$
and
$Q_{x, r}^{t, t + \varepsilon} \cap \Omega = \emptyset$.
The set 
$\Gamma (\Omega) = \partial \Omega \setminus \gamma (\Omega)$
is called the proper (or parabolic) boundary of $\Omega$.
\end{Definition}

\begin{Lemma}\label{L3.1}
Let $v$ and $w$ satisfy the inequalities
\begin{equation}
	L v - v_t - f (x, u, Dv)
	\ge
	L w - w_t - f (x, u, Dw)
	\quad
	\mbox{in } \omega \cup \gamma (\omega)
	\label{L3.1.1}
\end{equation}
and
\begin{equation}
	v \le w
	\quad
	\mbox{on } \Gamma (\omega),
	\label{L3.1.2}
\end{equation}
where $\omega \ne \emptyset$ is a bounded open subset of ${\mathbb R}_+^n$ 
such that $u$ is a positive function on $\omega \cup \gamma (\omega)$.
Then
\begin{equation}
	v \le w
	\quad
	\mbox{on } \omega \cup \gamma (\omega).
	\label{L3.1.3}
\end{equation}
\end{Lemma}

\begin{proof}
Lemma~\ref{L3.1} is the standard maximum principle for parabolic inequalities 
in bounded domains~\cite{Landis}.
The only subtlety is that~\eqref{L3.1.1} contains the function $f$.
However, this fact can not affect the proof in a significant way.
Not wanting to be unfounded, we give this proof in detail.

It can obviously be assumed that, in formula~\eqref{L3.1.1}, the inequality is strong;
otherwise we replace $v$ by $v - \varepsilon t$ 
and pass to the limit as $\varepsilon \to +0$.

Denote
$$
	\varphi = v - w.
$$
If~\eqref{L3.1.3} is not valid, then there exists a real number $\mu > 0$ for which the set
$\omega_\mu = \{ (x, t) \in \omega : \varphi (x, t) > \mu \}$ is not empty.

According to~\eqref{L3.1.2}, the closure $\overline{\omega}_\mu$ of the set $\omega_\mu$ 
is contained in $\omega \cup \gamma (\omega)$.
Let us take a point $(x', t') \in \overline{\omega}_\mu$ such that
\begin{equation}
	\varphi (x', t')
	=
	\sup_{
		\overline{\omega}_\mu
	}
	\varphi.
	\label{PL3.1.1}
\end{equation}
We have $D \varphi (x', t') = 0$ and $\varphi_t (x', t') \ge 0$ or, in other words,
$D v (x', t') = D w (x', t')$ and $v_t (x', t') \ge w_t (x', t')$.
It can easily be seen that
\begin{equation}
	\sum_{i,j=1}^n
	a_{ij} (x, u)
	\left.
		\frac{\partial^2 \varphi}{\partial x_i \partial x_j}
	\right|_{
		x = x',
		\,
		t = t'
	}
	\le
	0;
	\label{PL3.1.2}
\end{equation}
otherwise, introducing the new coordinates $y = y (x)$ such that
$$
	\sum_{i,j=1}^n
	a_{ij} (x, u)
	\left.
		\frac{\partial^2 \varphi}{\partial x_i \partial x_j}
	\right|_{
		x = x',
		\,
		t = t'
	}
	=
	\sum_{i=1}^n
	\left.
		\frac{\partial^2 \varphi}{\partial y_i^2}
	\right|_{
		y = y (x'),
		\,
		t = t'
	},
$$
we obtain
$$
	\left.
		\frac{\partial^2 \varphi}{\partial y_i^2}
	\right|_{
		y = y',
		\,
		t = t'
	}
	>
	0
$$
for some $1 \le i \le n$. This contradicts~\eqref{PL3.1.1}. 

At the same time,~\eqref{PL3.1.2} is equivalent to the relation
$$
	\sum_{i,j=1}^n
	a_{ij} (x, u)
	\left.
		\frac{\partial^2 v}{\partial x_i \partial x_j}
	\right|_{
		x = x',
		\,
		t = t'
	}
	\le
	\sum_{i,j=1}^n
	a_{ij} (x, u)
	\left.
		\frac{\partial^2 w}{\partial x_i \partial x_j}
	\right|_{
		x = x',
		\,
		t = t'
	};
$$
therefore, we arrive at a contradiction 
with our assumption that the inequality in~\eqref{L3.1.1} is strong.

The proof is completed.
\end{proof}

\begin{Corollary}\label{C3.1}
Let $v$ satisfy the inequality
$$
	L v - v_t - f (x, u, Dv)
	\ge
	0
	\quad
	\mbox{in } \omega \cup \gamma (\omega),
$$
where $\omega \ne \emptyset$ is a bounded open subset of ${\mathbb R}_+^n$ 
such that $u$ is a positive function on $\omega \cup \gamma (\omega)$.
Then
$$
	v (x, t)
	\le
	\sup_{
		\Gamma (\omega)
	}
	v
$$
for all $(x, t) \in \omega \cup \gamma (\omega)$.
\end{Corollary}

\begin{proof}
In Lemma~\ref{L3.1}, we take 
$$
	w
	=
	\sup_{
		\Gamma (\omega)
	}
	v.
$$
\end{proof}

\begin{Lemma}\label{L3.2}
Let $0 < r_1 < r_2$ and $0 < \tau_1 < \tau_2 < \tau$ be real numbers with
$$
	\sup_{
		Q_{r_1}^{\tau - \tau_1, \tau}
	}
	u
	>
	0.
$$
Then at least one of the following three estimates is valid:
$$
	M_2 - M_1
	\ge
	C
	(r_2 - r_1)^2
	q (r_2) 
	\lim_{\mu \to M_1 - 0}
	\inf_{	
		[\mu, M_2]
	}
	g,
$$
$$
	M_2 - M_1
	\ge
	C
	(r_2 - r_1) r_2
	q (r_2) 
	\lim_{\mu \to M_1 - 0}
	\inf_{	
		[\mu, M_2]
	}
	h,
$$
$$
	M_2 - M_1
	\ge
	C
	(\tau_2 - \tau_1)
	q (r_2) 
	\lim_{\mu \to M_1 - 0}
	\inf_{	
		[\mu, M_2]
	}
	g,
$$
where
$$
	M_1
	=
	\sup_{
		Q_{r_1}^{\tau - \tau_1, \tau}
	}
	u
	\quad
	\mbox{and}
	\quad
	M_2
	=
	\sup_{
		Q_{r_2}^{\tau - \tau_2, \tau}
	}
	u.
$$
\end{Lemma}

\begin{proof}
Assume that $\mu$ is a real number satisfying the condition
$
	0
	<
	\mu
	<
	M_1.
$
Also let $r_0 = (r_1 + r_2) / 2$ and $\varphi \in C^\infty ({\mathbb R})$ 
be a non-decreasing function  such that 
$$
	\left.
		\varphi 
	\right|_{
		(-\infty, 0]
	}
	= 
	0
	\quad
	\mbox{and}
	\quad
	\left.
		\varphi 
	\right|_{
		[1, \infty)
	}
	= 
	1.
$$ 
We denote
$$
	w (x, t)
	=
	k_1
	\varphi
	\left(
		\frac{
			|x| - r_0
		}{
			r_2 - r_0
		}
	\right)
	+
	k_2
	\varphi
	\left(
		\frac{
			\tau - \tau_1 - t
		}{
			\tau_2 - \tau_1
		}
	\right),
$$
where
$$
	k_1 
	= 
	\min 
	\left\{ 
		\frac{
			(r_2 - r_0)^2
			q (r_2)
			\inf_{	
				[\mu, M_2]
			}
				g
		}{
			2
			\| \varphi \|_{
				C^2 ([0, 1])
			}
		},
		\frac{
			(r_2 - r_0) r_2
			q (r_2) 
			\inf_{	
				[\mu, M_2]
			}
			h
		}{
			2
			\| \varphi \|_{
				C^1 ([0, 1])
			}
		}
	\right\}
$$
and
$$
	k_2
	=
	\frac{
		(\tau_2 - \tau_1)
		q (r_2) 
		\inf_{	
			[\mu, M_2]
		}
		g
	}{
		\| \varphi \|_{
			C^1 ([0, 1])
		}
	}.
$$
Further, let
$\Omega = \{ (x, t) \in {\mathbb R}_+^n  : u (x, t) > \mu \}$ 
and
$\omega = \Omega \cap Q_{r_2}^{\tau - \tau_2, \tau}$.
By direct calculation, it can be shown that
\begin{align}
	L w - w_t 
	=
	{}
	&
	k_1
	\varphi''
	\left(
		\frac{
			|x| - r_0
		}{
			r_2 - r_0
		}
	\right)
	\sum_{i.j=1}^n
	\frac{
		a_{ij} (x, u)
		x_i x_j
	}{
		|x|^2 (r_2 - r_0)^2
	}
	+
	k_1
	\varphi'
	\left(
		\frac{
			|x| - r_0
		}{
			r_2 - r_0
		}
	\right)
	\sum_{i=1}^n
	\frac{
		a_{ii} (x, u)
	}{
		|x| (r_2 - r_0)
	}
	\nonumber
	\\
	&
	-
	k_1
	\varphi'
	\left(
		\frac{
			|x| - r_0
		}{
			r_2 - r_0
		}
	\right)
	\sum_{i.j=1}^n
	\frac{
		a_{ij} (x, u)
		x_i x_j
	}{
		|x|^3 (r_2 - r_0)
	}
	+
	\frac{k_2}{\tau_2 - \tau_1}
	\varphi'
	\left(
		\frac{
			\tau - \tau_1 - t
		}{
			\tau_2 - \tau_1
		}
	\right)
	\nonumber
	\\
	\le
	{}
	&
	p (|x|)
	g (u)
	\left(
		1
		+
		\sum_{i,j=1}^n
		|a_{ij} (x, u)|
	\right)
	\label{PL3.2.2}
\end{align}
for all $(x, t) \in \omega \cup \gamma (\omega)$.
In so doing, we obviously have
$$
	D w 
	= 
	\frac{
		k_1 x
	}{
		|x| (r_2 - r_0)
	}
	\varphi'
	\left(
		\frac{
			|x| - r_0
		}{
			r_2 - r_0
		}
	\right)
$$
and
$$
	0
	\le
	\frac{
		k_1
	}{
		|x| (r_2 - r_0)
	}
	\varphi'
	\left(
		\frac{
			|x| - r_0
		}{
			r_2 - r_0
		}
	\right)
	\le
	p (|x|) h (u)
$$
for all $(x, t) \in \omega \cup \gamma (\omega)$,
whence in accordance with~\eqref{1.2} it follows that
$$
	f (x, u, D w)
	\ge
	p (|x|)
	g (u)
	\left(
		1
		+
		\sum_{i,j=1}^n
		|a_{ij} (x, u)|
	\right)
$$
for all $(x, t) \in \omega \cup \gamma (\omega)$.
Combining the last inequality with~\eqref{PL3.2.2}, we obtain
\begin{equation}
	L w - w_t
	\le 
	f (x, u, D w) 
	\quad
	\mbox{in } \omega \cup \gamma (\omega).
	\label{PL3.2.3}
\end{equation}
Let us show that
\begin{equation}
	\sup_{
		\overline{\omega}
		\cap
		\Gamma (Q_{r_2}^{\tau - \tau_2, \tau})
	}
	(u - w)
	\ge
	\sup_\omega
	{}
	(u - w).
	\label{PL3.2.4}
\end{equation}
Really, if~\eqref{PL3.2.4} is not valid, then
\begin{equation}
	\sup_{
		\overline{\omega}
		\cap
		\Gamma (Q_{r_2}^{\tau - \tau_2, \tau})
	}
	(u - w)
	<
	\sup_\omega
	{}
	(u - w)
	-
	\varepsilon
	\label{PL3.2.5}
\end{equation}
for some $\varepsilon > 0$.
Without loss of generality, it can also be assumed that
$
	\mu
	<
	M_1 - \varepsilon.
$
We denote
\begin{equation}
	v  
	= 
	u 
	-
	\sup_\omega
	{}
	(u - w)
	+
	\varepsilon.
	\label{PL3.2.6}
\end{equation}
From~\eqref{T2.2.1}, it follows that
$$
	L v - v_t \ge f (x, t, u, D v)
	\quad
	\mbox{in } \omega \cup \gamma (\omega).
$$
Combining this with~\eqref{PL3.2.3}, we immediately obtain~\eqref{L3.1.1}.
Let us now establish the validity of inequality~\eqref{L3.1.2}.
We have
\begin{equation}
	\Gamma (\omega) 
	\subset
	\left(
		\Gamma (\Omega) 
		\cap
		\overline{
			Q_{r_2}^{\tau - \tau_2, \tau}
		}
	\right)
	\cup
		\left(
		\overline{\omega}
		\cap
		\Gamma (Q_{r_2}^{\tau - \tau_2, \tau})
	\right)
	\label{PL3.2.7}
\end{equation}
and
\begin{equation}
	\left.
		u
	\right|_{
		\Gamma (\Omega) 
		\cap
		\overline{
			Q_{r_2}^{\tau - \tau_2, \tau}
		}
	}
	=
	\mu.
	\label{PL3.2.8}
\end{equation}
Relations~\eqref{PL3.2.6} and~\eqref{PL3.2.8} imply that
$$
	\left.
		v
	\right|_{
		\Gamma (\Omega) 
		\cap
		\overline{
			Q_{r_2}^{\tau - \tau_2, \tau}
		}
	}
	=
	\mu
	-
	\sup_\omega
	{}
	(u - w)
	+
	\varepsilon.
$$
In so doing,
$$
	\sup_\omega
	{}
	(u - w)
	\ge
	\sup_{
		\Omega \cap Q_{r_1}^{\tau - \tau_1, \tau}
	}
	u
	=
	M_1
$$
since 
$
	\Omega \cap Q_{r_1}^{\tau - \tau_1, \tau}
	\subset
	\omega
$
and $w$ is equal to zero on $Q_{r_1}^{\tau - \tau_1, \tau}$;
therefore, we obtain
$$
	\left.
		v
	\right|_{
		\Gamma (\Omega) 
		\cap
		\overline{
			Q_{r_2}^{\tau - \tau_2, \tau}
		}
	}
	\le
	\mu - M_1 + \varepsilon
	< 
	0.
$$
The last formula and the fact that $w$ is a non-negative function yield
$$
	v \le w
	\quad
	\mbox{on } 
	\Gamma (\Omega) 
	\cap
	\overline{
		Q_{r_2}^{\tau - \tau_2, \tau}
	}.
$$
At the same time, taking into account~\eqref{PL3.2.5} and~\eqref{PL3.2.6}, we have
$$
	\sup_{
		\overline{\omega}
		\cap
		\Gamma (Q_{r_2}^{\tau - \tau_2, \tau})
	}
	(v - w)
	=
	\sup_{
		\overline{\omega}
		\cap
		\Gamma (Q_{r_2}^{\tau - \tau_2, \tau})
	}
	(u - w)
	-
	\sup_\omega
	{}
	(u - w)
	+
	\varepsilon
	<
	0.
$$
Consequently, one can assert that~\eqref{L3.1.2} is fulfilled.
Thus, by Lemma~\ref{L3.1}, inequality~\eqref{L3.1.3} holds or, in other words,
$$
	u 
	-
	\sup_\omega
	{}
	(u - w)
	+
	\varepsilon
	\le
	w
	\quad
	\mbox{on } \omega \cup \gamma (\omega),
$$
whence it follows that
$$
	\sup_{
		\omega \cup \gamma (\omega)
	}
	(u - w)
	+
	\varepsilon
	\le
	\sup_\omega
	{}
	(u - w).
$$
This contradiction proves~\eqref{PL3.2.4}.

Since $w$ is equal to zero on $Q_{r_1}^{\tau - \tau_1, \tau}$, 
formula~\eqref{PL3.2.4} implies the estimate
$$
	\sup_{
		\overline{\omega}
		\cap
		\Gamma (Q_{r_2}^{\tau - \tau_2, \tau})
	}
	(u - w)
	\ge
	\sup_{
		\omega \cap Q_{r_1}^{\tau - \tau_1, \tau}
	}
	u
$$
from which, by the relations
$$
	\sup_{
		\overline{\omega}
		\cap
		\Gamma (Q_{r_2}^{\tau - \tau_2, \tau})
	}
	u
	-
	\inf_{
		\overline{\omega}
		\cap
		\Gamma (Q_{r_2}^{\tau - \tau_2, \tau})
	}
	w
	\ge
	\sup_{
		\overline{\omega}
		\cap
		\Gamma (Q_{r_2}^{\tau - \tau_2, \tau})
	}
	(u - w)
$$
and
$$
	\sup_{
		\omega \cap Q_{r_1}^{\tau - \tau_1, \tau}
	}
	u
	=
	M_1,
$$
we obtain
\begin{equation}
	\sup_{
		\overline{\omega}
		\cap
		\Gamma (Q_{r_2}^{\tau - \tau_2, \tau})
	}
	u
	-
	\inf_{
		\overline{\omega}
		\cap
		\Gamma (Q_{r_2}^{\tau - \tau_2, \tau})
	}
	w
	\ge
	M_1.
	\label{PL3.2.9}
\end{equation}
From Corollary~\ref{C3.1}, it follows that
$$
	\sup_{
		\Gamma (\omega)
	}
	u
	=
	\sup_{
		\overline{\omega}
	}
	u
	=
	M_2.
$$
In addition,~\eqref{PL3.2.8} implies the inequality
$$
	\left.
		u
	\right|_{
		\Gamma (\Omega) 
		\cap
		\overline{
			Q_{r_2}^{\tau - \tau_2, \tau}
		}
	}
	<
	M_2;
$$
therefore, inclusion~\eqref{PL3.2.7} allows us to assert that
$$
	\sup_{
		\overline{\omega}
		\cap
		\Gamma (Q_{r_2}^{\tau - \tau_2, \tau})
	}
	u
	=
	\sup_{
		\Gamma (\omega)
	}
	u
	=
	M_2.
$$
Combining this with~\eqref{PL3.2.9}, we obtain
$$
	M_2
	-
	M_1
	\ge
	\inf_{
		\Gamma (Q_{r_2}^{\tau - \tau_2, \tau})
	}
	w.
$$
To complete the proof, it remains to note that
$$
	\inf_{
		\Gamma (Q_{r_2}^{\tau - \tau_2, \tau})
	}
	w
	\ge
	\min \{ k_1, k_2 \}.
$$
\end{proof}

\begin{Lemma}\label{L3.3}
Let $r > 0$ and $t > 0$ be real numbers such that $4 r^2 < t$ and
\begin{equation}
	\sup_{
		Q_{2 r}^{t - 4 r^2, t}
	}
	u
	\ge
	\theta^{1/2}
	\sup_{
		Q_r^{t - r^2, t}
	}
	u
	>
	0.
	\label{L3.3.1}
\end{equation}
Then at least one of the following two estimates is valid:
\begin{equation}
	\int_{M_1}^{M_2}
	(g_\theta (\zeta) \zeta)^{- 1/2}
	\,
	d\zeta
	\ge
	C
	\int_r^{2 r}
	q^{1/2} (2 \rho)
	\,
	d\rho,
	\label{L3.3.2}
\end{equation}
\begin{equation}
	\int_{M_1}^{M_2}
	\frac{
		d\zeta
	}{
		h_\theta (\zeta)
	}
	\ge
	C
	\int_r^{2 r}
	\rho q (2 \rho)
	\,
	d\rho,
	\label{L3.3.3}
\end{equation}
where
\begin{equation}
	M_1
	=
	\sup_{
		Q_r^{t - r^2, t}
	}
	u
	\quad
	\mbox{and}
	\quad
	M_2
	=
	\sup_{
		Q_{2 r}^{t - 4 r^2, t}
	}
	u.
	\label{L3.3.4}
\end{equation}
\end{Lemma}

\begin{proof}
Let $k$ be the maximal positive integer for which $\theta^{k/2} M_1 \le M_2$.
We put 
$
	m_i 
	= 
	\theta^{i/2}
	M_1,
$
$
	i = 0, \ldots, k - 1,
$
and $m_k = M_2$. 
It can easily be seen that
$$
	\theta^{1/2} m_i \le m_{i+1} < \theta m_i,
	\quad
	i = 0, \ldots, k - 1.
$$
Let us further take an increasing sequence of real numbers 
$\{ r_i \}_{i=0}^k$ 
such that
$r_0 = r$, 
$r_k = 2r$, 
and
$$
	\sup_{
		Q_{r_i}^{t - r_i^2, t}
	}
	u
	=
	m_i,
	\quad
	i = 1, \ldots, k - 1.
$$
Since $u$ is a continuous function in ${\mathbb R}_+^n$, such a sequence obviously exists.

By Lemma~\ref{L3.2}, for any $i \in \{ 0, \ldots, k - 1 \}$ 
at least one of the following three inequalities is valid:
\begin{equation}
	m_{i+1} - m_i
	\ge
	C
	(r_{i+1} - r_i)^2
	q (r_{i+1}) 
	\lim_{\mu \to m_i - 0}
	\inf_{	
		[\mu, m_{i+1}]
	}
	g,
	\label{PL3.3.1}
\end{equation}
\begin{equation}
	m_{i+1} - m_i
	\ge
	C
	(r_{i+1} - r_i) r_{i+1}
	q (r_{i+1}) 
	\lim_{\mu \to m_i - 0}
	\inf_{	
		[\mu, m_{i+1}]
	}
	h,
	\label{PL3.3.2}
\end{equation}
\begin{equation}
	m_{i+1} - m_i
	\ge
	C
	(r_{i+1}^2 - r_i^2)
	q (r_{i+1}) 
	\lim_{\mu \to m_i - 0}
	\inf_{	
		[\mu, m_{i+1}]
	}
	g.
	\label{PL3.3.3}
\end{equation}
If~\eqref{PL3.3.1} is valid, then we have
$$
	\left(
		\frac{
			m_{i+1} - m_i
		}{
			\lim_{\mu \to m_i - 0}
			\inf_{	
				[\mu, m_{i+1}]
			}
			g
		}
	\right)^{1/2}
	\ge
	C
	(r_{i+1} - r_i)
	q^{1/2} (r_{i+1}).
$$
Since
$$
	\int_{
		m_i
	}^{
		m_{i+1}
	}
	(g_\theta (\zeta) \zeta)^{- 1/2}
	\,
	d\zeta
	\ge
	C
	\left(
		\frac{
			m_{i+1} - m_i
		}{
			\lim_{\mu \to m_i - 0}
			\inf_{	
				[\mu, m_{i+1}]
			}
			g
		}
	\right)^{1/2},
$$
this implies the estimate
\begin{equation}
	\int_{
		m_i
	}^{
		m_{i+1}
	}
	(g_\theta (\zeta) \zeta)^{- 1/2}
	\,
	d\zeta
	\ge
	C
	(r_{i+1} - r_i)
	q^{1/2} (r_{i+1}).
	\label{PL3.3.4}
\end{equation}
In turn, if~\eqref{PL3.3.2} holds, then
$$
	\frac{
		m_{i+1} - m_i
	}{
		\lim_{\mu \to m_i - 0}
		\inf_{	
			[\mu, m_{i+1}]
		}
		h
	}
	\ge
	C
	(r_{i+1} - r_i) 
	r_{i+1}
	q (r_{i+1}),
$$
whence in accordance with the inequality
$$
	\int_{
		m_i
	}^{
		m_{i+1}
	}
	\frac{
		d\zeta
	}{
		h_\theta (\zeta)
	}
	\ge
	\frac{
		C (m_{i+1} - m_i)
	}{
		\lim_{\mu \to m_i - 0}
		\inf_{	
			[\mu, m_{i+1}]
		}
		h
	}
$$
we obtain
\begin{equation}
	\int_{
		m_i
	}^{
		m_{i+1}
	}
	\frac{
		d\zeta
	}{
		h_\theta (\zeta)
	}
	\ge
	C
	(r_{i+1} - r_i) 
	r_{i+1}
	q (r_{i+1}).
	\label{PL3.3.5}
\end{equation}
Let us also note that~\eqref{PL3.3.3} implies~\eqref{PL3.3.1}; 
therefore, in this case, we again arrive at~\eqref{PL3.3.4}.
Thus, for any $i \in \{ 0, \ldots, k - 1 \}$ at least one of 
estimates~\eqref{PL3.3.4}, \eqref{PL3.3.5} is valid.
We denote by $\Xi_1$ the set of integers $i \in \{ 0, \ldots, k - 1 \}$ 
for which~\eqref{PL3.3.4} holds. 
In so doing, let $\Xi_2 = \{ 0, \ldots, k - 1 \} \setminus \Xi_1$.

At first assume that
\begin{equation}
	\sum_{i \in \Xi_1}
	(r_{i+1} - r_i) 
	\ge
	\frac{r_k - r_0}{2}.
	\label{PL3.3.6}
\end{equation}
Then, summing~\eqref{PL3.3.4} over all $i \in \Xi_1$, we have
$$
	\int_{
		M_1
	}^{
		M_2
	}
	(g_\theta (\zeta) \zeta)^{- 1/2}
	\,
	d\zeta
	\ge
	C
	(r_k - r_0)
	q^{1/2} (r_k).
$$
This implies~\eqref{L3.3.2}.

Now, let~\eqref{PL3.3.6} is not valid. Then
$$
	\sum_{i \in \Xi_2}
	(r_{i+1} - r_i) 
	\ge
	\frac{r_k - r_0}{2};
$$
therefore, summing~\eqref{PL3.3.5} over all $i \in \Xi_2$, we conclude that
$$
	\int_{
		M_1
	}^{
		M_2
	}
	\frac{
		d\zeta
	}{
		h_\theta (\zeta)
	}
	\ge
	C
	(r_k - r_0) 
	r_0
	q (r_k),
$$
whence~\eqref{L3.3.3} immediately follows.

The proof is completed.
\end{proof}

\begin{Lemma}\label{L3.4}
In the hypotheses of Lemma~$\ref{L3.3}$, let the inequality
$$
	\theta^{1/2}
	\sup_{
		Q_r^{t - r^2, t}
	}
	u
	\ge
	\sup_{
		Q_{2 r}^{t - 4 r^2, t}
	}
	u
	>
	0
$$
be fulfilled instead of~\eqref{L3.3.1}.
Then at least one of the following two estimates is valid:
\begin{equation}
	\int_{M_1}^{M_2}
	\frac{
		d\zeta
	}{
		g_{\sqrt{\theta}} (\zeta)
	}
	\ge
	C
	\int_r^{2 r}
	\rho q (2 \rho)
	\,
	d\rho,
	\label{L3.4.1}
\end{equation}
\begin{equation}
	\int_{M_1}^{M_2}
	\frac{
		d\zeta
	}{
		h_{\sqrt{\theta}} (\zeta)
	}
	\ge
	C
	\int_r^{2 r}
	\rho q (2 \rho)
	\,
	d\rho,
	\label{L3.4.2}
\end{equation}
where $M_1$ and $M_2$ are defined by~\eqref{L3.3.4}.
\end{Lemma}

\begin{proof}
By Lemma~\ref{L3.2}, we obviously obtain either
\begin{equation}
	M_2 - M_1
	\ge
	C
	r^2
	q (2 r) 
	\lim_{\mu \to M_1 - 0}
	\inf_{	
		[\mu, M_2]
	}
	g
	\label{PL3.4.1}
\end{equation}
or
\begin{equation}
	M_2 - M_1
	\ge
	C
	r^2
	q (2 r) 
	\lim_{\mu \to M_1 - 0}
	\inf_{	
		[\mu, M_2]
	}
	h.
	\label{PL3.4.2}
\end{equation}
If~\eqref{PL3.4.1} holds, then
$$
	\frac{
		M_2 - M_1
	}{
		\lim_{\mu \to M_1 - 0}
		\inf_{	
			[\mu, M_2]
		}
		g
	}
	\ge
	C
	r^2
	q (2 r).
$$
Thus, taking into account the inequality
$$
	\int_{M_1}^{M_2}
	\frac{
		d\zeta
	}{
		g_{\sqrt{\theta}} (\zeta)
	}
	\ge
	C
	\frac{
		M_2 - M_1
	}{
		\lim_{\mu \to M_1 - 0}
		\inf_{	
			[\mu, M_2]
		}
		g
	},
$$
we can assert that
$$
	\int_{M_1}^{M_2}
	\frac{
		d\zeta
	}{
		g_{\sqrt{\theta}} (\zeta)
	}
	\ge
	C
	r^2
	q (2 r),
$$
whence~\eqref{L3.4.1} follows at once.

Analogously,~\eqref{PL3.4.2} implies the estimate
$$
	\int_{M_1}^{M_2}
	\frac{
		d\zeta
	}{
		h_{\sqrt{\theta}} (\zeta)
	}
	\ge
	C
	r^2
	q (2 r)
$$
from which~\eqref{L3.4.2} can be obtained.

The proof is completed.
\end{proof}

\begin{Lemma} \label{L3.5}
Let
$0 < \alpha \le 1$,
$\sigma > 1$,
$\nu > 1$,
$M_1 > 0$,
and
$M_2 > 0$
be real numbers with
$M_2 \ge \nu M_1$.
Then
$$
	\left(
		\int_{M_1}^{M_2}
		\psi_\sigma^{-\alpha} (s)
		s^{\alpha - 1}
		\,
		ds
	\right)^{1 / \alpha}
	\ge
	A
	\int_{M_1}^{M_2}
	\frac{
		ds
	}{
		\psi (s)
	}
$$
for any measurable function $\psi : (0,\infty) \to (0,\infty)$ such that
$\psi_\sigma (s) > 0$ for all $s \in (0,\infty)$,
where $A > 0$ is a constant depending only on $\alpha$, $\nu$, and $\sigma$.
\end{Lemma}

\begin{Lemma} \label{L3.6}
Let
$0 < \alpha \le 1$,
$\sigma > 1$,
$\nu > 1$,
$r_1 > 0$,
and
$r_2 > 0$
be real numbers with
$r_2 \ge \nu r_1$.
Then
$$
	\left(
		\int_{r_1}^{r_2}
		\varphi^\alpha (r)
		\,
		dr
	\right)^{1 / \alpha}
	\ge
	A
	\int_{r_1}^{r_2}
	r^{1 / \alpha - 1}
	\varphi_\sigma (r)
	\,
	dr
$$
for any measurable function $\varphi : [r_1, r_2] \to [0,\infty)$,
where $A > 0$ is a constant depending only on $\alpha$, $\nu$, and $\sigma$.
\end{Lemma}

Lemmas~\ref{L3.5} and~\ref{L3.6} are proved in~\cite[Lemmas~2.3 and~2.6]{meIzv}.

\begin{proof}[Proof of Theorem~$\ref{T2.2}$]
Let $K$ be a compact subset of ${\mathbb R}^n$ and, moreover,
$r > 0$ and $t > 0$ be real numbers such that $K \subset B_r$ and $t > r^2$.
If
$$
	\sup_{x \in K}
	u (x, t)
	\le
	0,
$$
then
$$
	\sup_{x \in K}
	u_+ (x, t)
	=
	0;
$$
therefore, we can assume that
$$
	\sup_{x \in K}
	u (x, t)
	>
	0.
$$
Let us take the maximal integer $k$ satisfying the condition $4^k r^2 < t$.
Also put $r_i = 2^i r$ and
$$
	m_i
	=
	\sup_{
		Q_{r_i}^{t - r_i^2, t}
	}
	u,
	\quad
	i = 0, \ldots, k.
$$

We show that
\begin{equation}
	\left(
		\int_{
			m_0
		}^{
			\infty
		}
		(g_\theta (\zeta) \zeta)^{- 1/2}
		\,
		d\zeta
	\right)^2
	+
	\int_{
		m_0
	}^{
		\infty
	}
	\frac{
		d\zeta
	}{
		h_\theta (\zeta)
	}
	\ge
	C
	\int_{
		r_0
	}^{
		r_k
	}
	\rho
	q (4 \rho)
	\,
	d\rho.
	\label{PT2.2.1}
\end{equation}
Really, by Lemmas~\ref{L3.3} and~\ref{L3.4}, for any $i \in \{ 0, \ldots, k - 1 \}$ 
at least one of the following three inequalities is valid:
\begin{equation}
	\int_{
		m_i
	}^{
		m_{i+1}
	}
	(g_\theta (\zeta) \zeta)^{- 1/2}
	\,
	d\zeta
	\ge
	C
	\int_{
		r_i
	}^{
		r_{i+1}
	}
	q^{1/2} (2 \rho)
	\,
	d\rho,
	\label{PT2.2.2}
\end{equation}
\begin{equation}
	\int_{
		m_i
	}^{
		m_{i+1}
	}
	\frac{
		d\zeta
	}{
		h_\theta (\zeta)
	}
	\ge
	C
	\int_{
		r_i
	}^{
		r_{i+1}
	}
	\rho q (2 \rho)
	\,
	d\rho,
	\label{PT2.2.3}
\end{equation}
\begin{equation}
	\int_{
		m_i
	}^{
		m_{i+1}
	}
	\frac{
		d\zeta
	}{
		g_{\sqrt{\theta}} (\zeta)
	}
	\ge
	C
	\int_{
		r_i
	}^{
		r_{i+1}
	}
	\rho q (2 \rho)
	\,
	d\rho.
	\label{PT2.2.4}
\end{equation}
By $\Xi_1$, $\Xi_2$, and $\Xi_3$ we denote the sets of integers $i \in \{ 0, \ldots, k - 1 \}$ 
satisfying relations~\eqref{PT2.2.2}, \eqref{PT2.2.3}, and~\eqref{PT2.2.4}, respectively.

At first, let
\begin{equation}
	\sum_{
		i \in \Xi_1
	}
	\int_{
		r_i
	}^{
		r_{i+1}
	}
	\rho q (4 \rho)
	\,
	d\rho
	\ge
	\frac{1}{2}
	\int_{
		r_0
	}^{
		r_k
	}
	\rho q (4 \rho)
	\,
	d\rho.
	\label{PT2.2.5}
\end{equation}
Summing~\eqref{PT2.2.2} over all $i \in \Xi_1$, we obtain
$$
	\int_{
		m_0
	}^{
		\infty
	}
	(g_\theta (\zeta) \zeta)^{- 1/2}
	\,
	d\zeta
	\ge
	C
	\sum_{
		i \in \Xi_1
	}
	\int_{
		r_i
	}^{
		r_{i+1}
	}
	q^{1/2} (2 \rho)
	\,
	d\rho.
$$
By Lemma~\ref{L3.6}, this implies the estimate
$$
	\left(
		\int_{
			m_0
		}^{
			\infty
		}
		(g_\theta (\zeta) \zeta)^{- 1/2}
		\,
		d\zeta
	\right)^2
	\ge
	C
	\sum_{
		i \in \Xi_1
	}
	\left(
		\int_{
			r_i
		}^{
			r_{i+1}
		}
		q^{1/2} (2 \rho)
		\,
		d\rho
	\right)^2
	\ge
	C
	\sum_{
		i \in \Xi_1
	}
	\int_{
		r_i
	}^{
		r_{i+1}
	}
	\rho
	q (4 \rho)
	\,
	d\rho
$$
from which~\eqref{PT2.2.1} immediately follows.

Now, assume that~\eqref{PT2.2.5} is not valid. Then
\begin{equation}
	\sum_{
		i \in \Xi_2 \cup \Xi_3
	}
	\int_{
		r_i
	}^{
		r_{i+1}
	}
	\rho q (4 \rho)
	\,
	d\rho
	\ge
	\frac{1}{2}
	\int_{
		r_0
	}^{
		r_k
	}
	\rho q (4 \rho)
	\,
	d\rho.
	\label{PT2.2.6}
\end{equation}
In this case, summing~\eqref{PT2.2.3} over all $i \in \Xi_2$, we have
\begin{equation}
	\int_{
		m_0
	}^{
		\infty
	}
	\frac{
		d\zeta
	}{
		h_\theta (\zeta)
	}
	\ge
	C
	\sum_{i \in \Xi_2}
	\int_{
		r_i
	}^{
		r_{i+1}
	}
	\rho q (2 \rho)
	\,
	d\rho.
	\label{PT2.2.7}
\end{equation}
Analogously,~\eqref{PT2.2.4} allows us to assert that
$$
	\int_{
		m_0
	}^{
		\infty
	}
	\frac{
		d\zeta
	}{
		g_{\sqrt{\theta}} (\zeta)
	}
	\ge
	C
	\sum_{i \in \Xi_3}
	\int_{
		r_i
	}^{
		r_{i+1}
	}
	\rho q (2 \rho)
	\,
	d\rho.
$$
Since
$$
	\left(
		\int_{
			m_0
		}^{
			\infty
		}
		(g_\theta (\zeta) \zeta)^{- 1/2}
		\,
		d\zeta
	\right)^2
	\ge
	C
	\int_{
		m_0
	}^{
		\infty
	}
	\frac{
		d\zeta
	}{
		g_{\sqrt{\theta}} (\zeta)
	}
$$
in view of Lemma~\ref{L3.5}, this yields the inequality
\begin{equation}
	\left(
		\int_{
			m_0
		}^{
			\infty
		}
		(g_\theta (\zeta) \zeta)^{- 1/2}
		\,
		d\zeta
	\right)^2
	\ge
	C
	\sum_{i \in \Xi_3}
	\int_{
		r_i
	}^{
		r_{i+1}
	}
	\rho q (2 \rho)
	\,
	d\rho.
	\label{PT2.2.8}
\end{equation}
Combining~\eqref{PT2.2.6}, \eqref{PT2.2.7}, and~\eqref{PT2.2.8}, 
we again arrive at~\eqref{PT2.2.1}.

Further, assuming that $r$ is fixed, we obviously obtain $r_k \to \infty$ as $t \to \infty$; 
therefore, in accordance with~\eqref{T2.1.1}, \eqref{T2.1.2}, and~\eqref{T2.1.3} 
formula~\eqref{PT2.2.1} implies that $m_0 \to 0$ as $t \to \infty$. Thus,
$$
	\sup_{
		Q_{r}^{t - r^2, t}
	}
	u_+
	\to
	0
	\quad
	\mbox{as } t \to \infty.
$$

The proof is completed.
\end{proof}

\end{document}